\documentclass[10pt]{article}

\usepackage{fullpage}
\usepackage{enumerate}

\usepackage{amsmath,amsfonts,amssymb,graphics,amsthm, comment}
\usepackage{hyperref}
\hypersetup{
    colorlinks=true,
    linkcolor=blue,
    citecolor=red,
    urlcolor=blue,
    pdfborder={0 0 0}
}

\usepackage{fullpage}

\usepackage{cite}
\usepackage{mathtools}
\usepackage{dsfont}
\usepackage{microtype}
\usepackage{xspace}

\usepackage{pdfsync}

\usepackage{graphicx}
\usepackage{xcolor}
\usepackage{bbm}
\usepackage{wrapfig} % for figures inserted on the side of the text

\usepackage[font=sf, labelfont={sf,bf}, margin=1cm]{caption}

\usepackage{cleveref}
  \crefname{theorem}{Theorem}{Theorems}
  \crefname{thm}{Theorem}{Theorems}
  \crefname{lemma}{Lemma}{Lemmas}
  \crefname{lem}{Lemma}{Lemmas}
  \crefname{remark}{Remark}{Remarks}
  \crefname{prop}{Proposition}{Propositions}
\crefname{notation}{Notation}{Notations}
\crefname{claim}{Claim}{Claims}
  \crefname{defn}{Definition}{Definitions}
  \crefname{corollary}{Corollary}{Corollaries}
  \crefname{section}{Section}{Sections}
  \crefname{figure}{Figure}{Figures}
    \crefname{assumption}{Assumption}{Assumptions}

\newtheorem{thm}{Theorem}[section]

\newtheorem{lemma}[thm]{Lemma}
\newtheorem{corollary}[thm]{Corollary}
\newtheorem{prop}[thm]{Proposition}
\newtheorem{defn}[thm]{Definition}

\numberwithin{equation}{section}

\theoremstyle{definition}
\newtheorem{remark}[thm]{Remark}

\newcommand{\mb}[1]{\mathbb{#1}}

\newcommand{\grad}{\nabla}
\newcommand{\ind}{\mathds{1}}
\newcommand{\hcap}{\mathrm{hcap}}
\newcommand{\height}{\mathrm{height}}
\newcommand{\harm}{\mathrm{Harm}_{\mb H}}
\newcommand{\harmt}{\mathrm{Harm}_{H_t}}
\newcommand{\harmh}{\mathrm{Harm}_H}
\newcommand{\harmd}{\mathrm{Harm}_{D}}
\newcommand{\harmdi}{\mathrm{Harm}_{\mb D}}
\newcommand{\ipd}[2]{ \langle #1, \, #2 \rangle }
\newcommand{\dipd}[2]{ \ipd{#1}{#2}_\grad }

\def\cF{\mathcal{F}}

\def\cD{\mathcal{D}}

\def\cA{\mathcal{A}}

\def\P{\mathbb{P}}
\def\E{\mathbb{E}}
\def\C{\mathbb{C}}
\def\R{\mathbb{R}}

\def\H{\mathbb{H}}

\def  \p- {p\textunderscore}

\def\eps{\varepsilon}

\DeclareMathOperator{\var}{Var}

\DeclareMathOperator{\dist}{dist}

\title{The Rohde--Schramm theorem, via the Gaussian free field}
\author{Nathana\"el Berestycki\thanks{The first author's research was supported in part by EPSRC grants EP/L018896/1 and EP/I03372X/1. The second author was a PhD student while this work was taking place, funded by EPSRC grant EP/H023348/1. }  \and Henry Jackson$^*$}
\date{\today}

\begin{document}
\maketitle

\begin{abstract}
The Rohde--Schramm theorem states that Schramm--Loewner Evolution with parameter $\kappa$ (or SLE$_\kappa$ for short) exists as a random curve, almost surely, if $\kappa \neq 8$. Here we give a new and concise proof of the result, based on the Liouville quantum gravity coupling (or reverse coupling) with a Gaussian free field. This transforms the problem of estimating the derivative of the Loewner flow into estimating certain correlated Gaussian free fields. While the correlation between these fields is not easy to understand, a surprisingly simple argument allows us to recover a derivative exponent first obtained by Rohde and Schramm \cite{RohdeSchramm}, subsequently shown to be optimal by Lawler and Viklund \cite{LawlerViklund}, which then implies the Rohde--Schramm theorem.
\end{abstract}

%The first shown in \cite{RohdeSchramm}. Our proof uses the coupling of the reverse SLE with the Neumann boundary GFF to bound the derivative of the inverse of the Loewner flow close to the origin. We are able to write the absolute value of the derivative of the reverse flow in terms of the exponential of two Gaussian random fields; one of them is a Neumann boundary GFF and the other is very closely related. We can use our knowledge of the structure of the GFF to find bounds which ensure continuity of the SLE trace.

\section{Introduction}
%We will provide a new proof for the Rohde-Schramm theorem in the case $\kappa \neq 8$. The theorem roughly states that ``a Schramm-Loewner evolution is generated by a path''. It was first proved for $\kappa \neq 8$ in \cite{RohdeSchramm} and for $\kappa = 8$ in \cite{lawler2011conformal}.

\subsection{Main result}

We start by recalling some basic definitions and notations. Let $(\xi_t)_{t\ge 0}$ be a real-valued continuous function. Define the family of conformal maps $(g_t)_{t\ge 0}$ as the maximal solution to Loewner's equation:
\begin{equation}
	dg_t(z) = \frac{2}{g_t(z) - \xi_t}dt
	\label{rs:eq:loewner}
\end{equation}
for each $z \in \mb H: = \{ z \in \C: \Im(z) >0 \}$, the upper-half plane. The maximal solution is defined up to some maximal time $t= \zeta(z)$ such that that $|g_t(z) - \xi_t| \to 0$ as $t \to \zeta(z)$.
%and we call that time
%\begin{equation*}
%	\tau(z) = \inf\left\{ t > 0 \; : \; g_t(z) = \xi_t \right\}.
%\end{equation*}
Set
\begin{equation*}
	H_t = \left\{ z \in \mb H \; : \; t < \zeta(z) \right\}.
\end{equation*}
The complement of $H_t$ in the upper half plane, $K_t = \mb H \setminus H_t$, is a compact $\H$-hull. It is easy to see that $g_t$ is a conformal isomorphism from $H_t$ to $\H$, which maps out the hull $K_t$. Deterministic Loewner theory (see \cite{Ahlfors, Lawler-book, SLEnotes}) implies that $K_t$ is a growing sequence satisfying the local growth property; and $(\xi_t)_{t\ge 0}$ is called the Loewner transform or driving function of the hulls process $(K_t)_{t\ge 0}$.

%The driving function $(\xi_t)$ is called the Loewner transform of the hulls $(K_t)$.

\medskip We say that the hulls $(K_t)_{t\ge 0 }$ are \textbf{generated by a curve} if there exists a curve $(\gamma_t)_{t \ge 0} \subset \bar \H$ such that, for all $t > 0$, $H_t$ is the (unique) unbounded component of $\mb H \setminus \gamma[0,t]$. The following, which appears as Theorem 4.1 in \cite{RohdeSchramm}, gives a condition for the hulls $(K_t)_{t\ge 0}$ to be generated by a curve.
\begin{thm}
	Let $\xi:[0,\infty) \to \mb R$ be continuous, and let $g_t$ be the corresponding Loewner flow, i.e., the solution of Loewner's equation \eqref{rs:eq:loewner}. Assume that %]
	\begin{equation*}
		\gamma(t) := \lim_{y \to 0} g_t^{-1}(iy + \xi_t)
	\end{equation*}
	exists for all $t \in [0,\infty)$ and is continuous. Then $g_t^{-1}$ extends continuously to $\overline{\mb H}$ and $H_t$ is the unbounded connected component of $\mb H \setminus \gamma[0,t]$, for every $t \in [0,\infty)$.%]]
		\label{rs:thm:continuity}
\end{thm}

By definition, the \textbf{(chordal) SLE$_\kappa$ process} is the growing process of hulls whose Lowener transform is given by $\xi_t = \sqrt \kappa B_t$, $t\ge 0$, where $\kappa > 0$ and $(B_t)_{t\ge 0}$ is a standard (linear) Brownian motion. We will assume throughout that $(\xi_t)_{t\ge 0}$ has this form. We also introduce the following notation:
\begin{equation}
	f_t := g_t^{-1} \qquad \text{ and } \qquad \hat f_t(z) := f_t(z + \xi_t) = g_t^{-1}(z+\xi_t).
	\label{rs:eq:fdefn}
\end{equation}

The celebrated \textbf{Rohde--Schramm theorem} \cite[Theorem 3.6]{RohdeSchramm}, shows that the assumption in Theorem \ref{rs:thm:continuity} is satisfied for $\kappa \neq 8$, and so SLE$_\kappa$ is generated by a curve for those values of $\kappa$:
\begin{thm}[Rohde--Schramm]\label{t:RS}
With the same notations as above, define
	\begin{equation*}
		H(y,t) := \hat f_t(iy) \quad \text{ for } \quad (y,t) \in (0,\infty) \times [0,\infty).
	\end{equation*}
	If $\kappa \neq 8$, then almost surely $H(y,t)$ extends continuously to $[0,\infty)\times[0,\infty)$. %]]]
	\label{rs:thm:hcts}
\end{thm}

A natural approach for proving convergence of $\hat f_t(iy)$ as $y \to 0$ is to show that it is Cauchy and hence, together with fairly standard distortion estimates (see Lemma 4.32 in \cite{Lawler-book}), it is sufficient to control the derivative $\hat f_t'(iy)$. In particular, the heart of the proof of \cref{t:RS} is the following bound on the tail of $|\hat f'_t(iy)|$:

\begin{thm}\label{T:main}
	Let $\kappa \neq 8$, and let $\hat f_t $ be the centred inverse of the Loewner flow as defined in \eqref{rs:eq:fdefn}. Then for all $\eps>0$, there exist constants $\varepsilon > 0$, $\delta > 0$ and $C > 0$ such that
	\begin{equation*}
		\mb P \left[ |\hat f'_t(iy)| > y^{-(1-\varepsilon)} \right] \leq C y^{2 + \delta}
	\end{equation*}
	for all $t \in [0,1]$ and $y \in (0,1)$.
	\label{rs:thm:tail}
\end{thm}

The main result of this paper is a new proof of the following more precise bound:

\begin{thm}\label{T:mainq}
  For any $\kappa \neq 8$, and for any $\delta >0$ there exist $C>0$ and $\eps>0$ such that for all $t \in [0,1]$ and $y \in (0,1)$,
  	\begin{equation*}
		\mb P \left[ |\hat f'_t(iy)| > y^{-(1-\varepsilon)} \right] \leq C y^{q - \delta}, \quad \quad \text{ where } \quad \quad q = \frac4{\kappa} + \frac{\kappa}{16} +1.
\end{equation*}
\end{thm}

Up to a change of notation, this exponent is the same as the one derived in \cite{RohdeSchramm} (and note that $q>2$ unless $\kappa = 8$). See also Theorem 7.4 in \cite{Lawler-book} which summarises the argument of \cite{RohdeSchramm} (keeping in mind that $a = 2/\kappa$ in Lawler's notations). This exponent was subsequently shown to be sharp by Lawler and Viklund \cite{LawlerViklund}: see e.g. (4.2), observing that the set of times such as $|f'_t(iy) | \approx y^{-1}$ corresponds in their notations to $\Theta_\beta$ with $\beta = 1$, in which case the value of $\rho$ is precisely what we call $q$ here.
The main theorem of Lawler and Viklund is to compute the Hausdorff dimension of $\Theta_\beta$. The case $\beta=1$ is of course excluded from that theorem, since in fact the set of such times is empty, but formally the dimension is $1- q/2$. In fact it is likely that the method of our proof could be used to compute the upper bound on the Hausdorff dimension of $\Theta_\beta$ for general values of $\beta$, but we have not tried in order to keep the paper as short and self-contained as possible. (It's also likely that a more elaborate version of this argument would also give the corresponding lower bound.)

We recall that it is also the case that SLE$_8$ is a.s. generated by a curve, but this follows from results of \cite{LSW} where it is shown that the SLE$_8$ is the scaling limit of the contour line of a uniform spanning tree with appropriate boundary condition.

\subsection{Main idea of proof} We present a simplified overview of the argument used to prove \cref{T:main}. The main idea for the proof of the paper is to exploit the coupling between reverse SLE and Neumann Gaussian free field (also known as \textbf{Liouville quantum gravity coupling}, see \cite{LQGnotes}) discovered by Sheffield \cite{zipper}. These objects will be introduced more precisely in the next section. In this coupling we find that if $f_t: \H \to H_t$ is the conformal map from the half-plane to the complement of the hull, if $h$ is a GFF (with Neumann boundary conditions on $\H$) and $h_0 = h + (2/\sqrt{\kappa}) \log |\cdot |$, then
\begin{equation}\label{E:rc}
h_t: = h_0 \circ f_t + Q \log | f_t'| ,
\end{equation}
viewed as a distribution in $\H$ modulo constants, has the same law as $h_0$. (This can be interpreted as expressing the domain Markov property of a certain quantum surface, see e.g. \cite{LQGnotes}). \eqref{E:rc} allows us to write the logarithm of the derivative as the difference of two GFFs:
\begin{equation}\label{E:rc2}
Q \log | f_t'| = h_t - h_0 \circ f_t
\end{equation}
We wish to estimate moments of $|f_t'|$. Hence we would like to exponentiate \eqref{E:rc2}. While the correlation between the two fields is poorly understood (indeed, this encodes all of SLE!) we can use H\"older's inequality and optimise over the choice of powers. It turns out that, despite being an incredibly simple and naive strategy, this is essentially enough to obtain the derivative exponent of \cref{T:main}.

There are a couple of additional points one needs to take care of in order to make this idea precise. First, the fields appearing in the right hand side of \eqref{E:rc2} are too rough to be made sense of pointwise. However, since the left hand side is a harmonic function, we can replace the fields in the right hand side by the harmonic extension of their boundary data on the real line. This replaces two rough fields by two nice harmonic functions. (That considering the harmonic extension should be sufficient for proving the result should not be too surprising: indeed, by Sheffield's quantum zipper theorem \cite{zipper}, the SLE curve can be obtained by conformally welding the upper half-plane to itself along the boundary, hence the boundary data of $h_0$ encodes all there is to know about SLE). Second, the equality \eqref{E:rc} (and hence \eqref{E:rc2} too) is only valid as distribution modulo constants, so one needs to track down the possibly random constant that one needs to add to make \eqref{E:rc2} true pointwise as harmonic functions. This is done by carefully choosing the normalisation of the GFF.

\subsection{Relation to other works}

The Liouville quantum gravity or reverse coupling has already been used in order to study SLE and in particular to obtain derivative estimates. This was first used by Miller and Sheffield in \cite{QLE} to show that the boundary of a QLE (quantum Loewner evolution) defines a H\"older domain and it is pointed out there that the same argument applies for SLE. It is also used in a paper by Gwynne, Miller and Sun \cite{GwynneMillerSun} on the almost sure multifractal spectrum of SLE to estimate the behaviour of the derivative of the reverse Loewner flow for an SLE$_{\kappa}(\rho)$ process near the force point. Lastly it was used by Miller and Sheffield in their remarkable work \cite{LQGandTBMII} on the relation between the Brownian map and Liouville quantum gravity, where it was the main ingredient to show that the QLE metric (in the case $\gamma = \sqrt{8/3}$) is a.s. homeomorphic to the Euclidean metric.

In all those works the difficulty is to handle the correlation between the Gaussian free fields coming from the LQG coupling. In \cite{QLE} and \cite{LQGandTBMII} this is done by considering the worst case scenarios for one or both fields, essentially replacing one of the field by its maximum. Clearly this is not optimal and one cannot hope to obtain sharp derivative exponents (as is the case here) from this technique. In \cite{GwynneMillerSun}, the LQG coupling is used as a step to obtain very precise estimates, but the actual application of the LQG coupling is only used to determine the expected behaviour of the derivative map in a reverse SLE$_{\kappa}(\rho)$ near the force point, hence the problem of correlation does not arise as such. Compared to those works, the main innovation of the present paper is to obtain an efficient way to handle the correlation between the fields, based on using H\"older's inequality, and looking at the harmonic extensions of the fields rather than the fields themselves. However simple, this naive idea is enough to recover sharp derivative exponents, which is perhaps surprising. It is likely that this idea could be useful in more complicated contexts.

We also point out that a more analytic approach  to the Rohde--Schramm theorem, based on rough paths theory, was also recently proposed by Friz and Shekhar in \cite{FrizShekhar}. Their argument has the advantage that it is more robust to perturbations in the driving Brownian motion, allowing in particular non constant values of $\kappa$. There is also a substantial literature on establishing deterministic versions of the Rohde--Schramm theorem (where one assumes some deterministic regularity condition on the driving function, which however is typically only satisfied by Brownian motion in an average sense, and not almost surely), using different techniques based on quasiconformal maps. See for instance \cite{MarshallRohde}, \cite{Lind} and e.g. \cite{RohdeTranZinsmeister} for recent developments trying to unify the techniques for deterministic and random driving functions.

\paragraph{Acknoweldgements.} We express our thanks to Christophe Garban, Jason Miller, James Norris and Fredrik Viklund for some enlightening discussions. We thank Huy Tran for pointing out a minor mistake in an earlier version and for calling our attention to the results in \cite{RohdeTranZinsmeister}.

\section{Preliminaries}

We now define carefully the objects needed to prove the theorem. We start with the reverse SLE flow, and follow up with some elementary properties of the Neumann GFF. These will be well known to experts in the area; however we have not found a place where they are written carefully. Also some of these depend subtly on the choice of normalisation (additive constant) for the Neumann GFF, so we spend some time writing them precisely.

\subsection{Reverse SLE}\label{subsec:reverse}
The time reversibility of the driving Brownian motion in \eqref{rs:eq:loewner} allows us to give a meaning to
growing a SLE ``backwards". Unlike in \eqref{rs:eq:loewner}, the hull will grow not from the tip but from the ``root", so unusual increments in the driving Brownian motion will be reflected by an unusual geometry of the hull near the ``root".
%The reverse SLE is the one we will need to use for the coupling between the Gaussian free field and SLE which we use, first shown in \cite{sheffield2010conformal}. We will explain more about this coupling in Section \ref{subsec:coupling}.
In order to follow the notation of \cite{zipper} it will be useful to fix the ``root" of the hull to be at the origin. This differs slightly from the definitions and notations given in say \cite{Lawler-book} (we will only consider what Lawler calls centered maps).

%We will explain how the different normalisation relates to that given in Definition \ref{defn:forwardsle} and the maps $(\hat{f}_t)$ used in Lemma \ref{lem:relations}.
\begin{defn}[Reverse SLE]\label{defn:reversesle}
	Fix  $(\xi_t)_{t\ge 0}$ a continuous function with $\xi_0 = 0$. For each $z \in \mb H$ let $f_t(z)$ be the solution to
\begin{equation}
	df_t(z) = -\frac{2}{f_t(z)}dt -  d\xi_t %\sqrt \kappa dB_t,
	\label{eq:reverse}
\end{equation}
with $f_0(z) = z$.
\end{defn}

It is easy to check (see e.g. \cref{L:im_rev} below)
that for a fixed $z \in \H$, $\Im( f_t(z))$ is now an increasing function of time which remains finite at any time $t>0$ and hence (unlike in the forward case) the solutions to \eqref{eq:reverse} are well defined for all times $t>0$, and $f_t(z)$ exists as a strong solution (Theorem 2.5 in \cite{KaratzasShreve}). We will call the collection of conformal maps $(f_t)$ a \textbf{reverse Loewner flow driven by $\xi$}. When $\xi$ is a Brownian motion with diffusivity $\kappa$ we will call the resulting family of conformal maps a \textbf{reverse SLE$_\kappa$ flow.} We will also use the notation $H_t = f_t( \H)$ and $K_t = \H \setminus K_t$ in this case, and will call $K_t$ the hull generated by the reverse Loewner flow.

We now show how \cref{defn:reversesle} relates to the maps $(\hat f_t)$ defined in \eqref{rs:eq:fdefn}.

\begin{lemma}\label{lem:relations}
	Fix a time $T> 0$. Let $(g_t)$ be a (forward) Loewner flow with driving function $(\xi_t)$, and let
	\begin{equation*}
		\hat{f}_t(z) := g_t^{-1}(z + \xi_t).
	\end{equation*}
	Let $(f_t)$ denote a reverse Loewner flow driven by $\hat \xi_s = \xi_T = \xi_{T-t}$. Then
		\begin{equation*}
		\hat{f}_T  = f_T.
	\end{equation*}
	In particular, in the SLE case, if $g_t$ is the forward SLE flow and $\hat f_t(z) = g_t^{-1}( z + \xi_t)$ is the centered inverse map, and if $(f_t)$ is a reverse SLE flow, we have for a fixed time $T>0$, $f_T = \hat f_T$ in distribution.
\end{lemma}
Note that the equality in distribution holds for a single fixed time $T$, not for the range of times $t \in [0,T]$.

\begin{proof}
This is largely well known, see e.g. Section 5.1 in \cite{LawlerViklund}. Fix $T>0$ and for $0\le s\le T$, set
$$
r_s = g_{T-s} \circ g_T^{-1},
$$
where $(g_t)$ is the (forward) Loewner flow driven by $(\xi_t)$, as in \eqref{rs:eq:loewner}. In words, the conformal map $r_s$ builds $g_{t-s} (K_t \setminus K_s)$, and maps $\H$ to the complement of $g_{t-s} (K_t \setminus K_s)$. Note that $r_0 = g_T \circ g_T^{-1} $ is the identity, while $r_T = g_T^{-1}$. Applying the change of variable $t \mapsto T-t$ in \eqref{rs:eq:loewner}, we see that $r_s$ solves the equation
$$
dr_s = \frac{- 2 ds}{r_s(z) - \xi_{T-s}},
$$
so that if $\hat r_s(z) = r_s(z + \xi_T) - \xi_{T-s}$, and $\hat \xi_s = \xi_T - \xi_{T-s}$, we have
$$
d\hat r_s( z) = \frac{-2}{\hat r_s(z)} ds - d\hat \xi_s, \quad 0 \le s\le T; \quad \quad \hat r_0(z) = z.
$$
In other words, $(\hat r_s)_{0 \le s \le T}$ is the reverse Loewner flow driven by $(\hat \xi_s)_{0 \le s \le T}$. In the case where $\xi$ has the law of a Brownian motion with diffusivity $\kappa$, so does $\hat \xi$ by time-reversibility of Brownian motion, and hence $(\hat r_s)_{0\le s \le T}$ has the law of a reverse SLE$_\kappa$. But since $\hat r_T = \hat f_T$, the lemma is proved.
\end{proof}

The following (well known) bounds will be useful for us later on.

\begin{lemma}\label{L:im_rev}
	Let $\kappa \geq 0$ and let $(f_t)$ be a reverse SLE$_\kappa$. Then, for any fixed $y > 0$, the imaginary part of $f_t(iy)$ increasing but bounded above for all $t \geq 0$ by
	\begin{equation*}
		\Im(f_t(iy)) \leq \sqrt{4t + y^2}.
	\end{equation*}
	\label{lem:im.bound}
\end{lemma}

\begin{proof}
	Fix $y > 0$ and write $f_t(iy) = Z_t = X_t + iY_t$. Then we know that $Z_t$ satisfies the SDE \eqref{eq:reverse} with $Z_0 = iy$. Taking the imaginary part of the equation we get
	\begin{align}
		\nonumber dY_t %&= - \Im\left( \frac{2}{Z_t} \right) dt \\
%		\nonumber &= -\Im\left( \frac{2(X_t - iY_t)}{|Z_t|^2} \right) dt \\
		&= \frac{2Y_t}{X_t^2 + Y_t^2}dt,
		\label{eq:y.sde}
	\end{align}
	with $Y_0 = y$. Since $y > 0$, we can see that the right hand side stays positive. In particular $Y_t$ is increasing.	
	%
	%The fact that $Y_t$ is increasing also shows us that $Z_t$ is bounded away from $0$, since $\Im(Z_t) \geq y_0 > 0$. That lets us see that the coefficients of the SDE which $Z$ satisfies,
%	\begin{equation*}
%		dZ_t = -\frac{2}{Z_t} dt - \sqrt \kappa dB_t,
%	\end{equation*}
%	are Lipschitz in space. Therefore, $Z_t$ exists as a strong solution. See Theorem 2.5 of \cite{karatzas1991brownian}, for example.
%
Moreover, since $X_t^2 \ge 0$, $dY_t \le (2/Y_t)dt$.
	%\begin{align*}
%		\nonumber dY_t %&= \frac{2Y_t}{X_t^2 + Y_t^2}dt \\
%		\nonumber &\leq \frac{2Y_t}{Y_t^2}dt\\
%		&\le \frac{2}{Y_t}dt.
%	\end{align*}
	Hence $ d(Y_t^2) \le 4 dt$. Integrating gives us the result.
	\end{proof}

\begin{corollary}\label{cor:fprime.bound}
	Let $\kappa \geq 0$ and let $(f_t)$ be a reverse SLE$_{\kappa}$. Then, for any fixed $y > 0$, the absolute value of the derivative, $|f'_t(iy)|$, is bounded above by
	\begin{equation*}
		|f'_t(iy)| \leq \frac{4}{y}  \sqrt{4t + y^2}.
	\end{equation*}
\end{corollary}
\begin{proof}
	As $f_t$ is a reverse SLE flow we know that $f_t: \mb H \to H_t$ for $H_t = \mb H \setminus K_t$. By Koebe's 1/4 theorem (Theorem 3.17 in \cite{Lawler-book}),
	\begin{equation}
		|f_t'(iy)| \dist(iy, \partial \mb H) \leq 4 \cdot \dist(f_t(iy), \partial H_t).
		\label{eq:kob}
	\end{equation}
	We know that $\dist(iy, \mb H) = y$ and
	\begin{align*}
		\dist(f_t(iy), \partial H_t) &\leq \dist(f_t(iy), \mb H) = \Im(f_t(iy)).
	\end{align*}
	Therefore we can use the bound from Lemma \ref{lem:im.bound} along with \eqref{eq:kob} to see
	\begin{equation*}
		|f'_t(iy)|y \leq 4\sqrt{4t + y^2}.
	\end{equation*}
	Rearranging gives the result.
\end{proof}

\begin{lemma}
	Let $\kappa \geq 0$, let $(f_t)$ be a reverse SLE$_\kappa$, and let $(\xi_t)$ be its driving process. Then, for any fixed $y > 0$, the absolute value of the real part of $f_t(iy)$ is bounded above by the absolute value of $\xi_t$, i.e.
	\begin{equation*}
		|\Re(f_t(iy))| \leq 2 \sup_{s\le t}|\xi_s|.
	\end{equation*}
	\label{lem:real.bound}
\end{lemma}

\begin{proof}
	As before, we write $f_t(iy) = Z_t = X_t + iY_t$ so that $\Re(f_t(iy)) = X_t$. Note that
	\begin{align*}
		dX_t %&
		%= -\Re\left( \frac{2}{Z_t} \right)dt - \sqrt \kappa dB_t \\
		%&= -\Re\left( \frac{2(X_t - iY_t)}{|Z_t|^2} \right)dt - \sqrt \kappa dB_t \\
		&= -\frac{2 X_t}{X_t^2 + Y_t^2} dt -  d\xi_t,
	\end{align*}
	with $X_0 = 0$. The drift term is always towards the origin (i.e. has a sign which is opposite of $X_t$). Consider the last time $S$ before $t$ such that $X_S =0$, so the sign of $X$ remains constant on $[S,t]$. Say without loss of generality that $X$ is nonnegative on that interval. Then 
$$
X_t = \int_S^t  \frac{ - 2 X_s}{X_s^2 + Y_s^2} ds  - (\xi_t - \xi_S)\le 2 \sup_{s\le t} |\xi_s|,
$$
as desired. 
\end{proof}

\begin{remark}
  The factor 2 can be removed in the above upper bound, see Lemma 2.1 in \cite{RohdeTranZinsmeister}. Also, a more elaborate argument based on Yamada's comparison theorem for solutions of stochastic differential equations (Theorem 1.1 in \cite{Yamada}) can be used to show that $X_t^2 \le \tilde \xi_t^2$ a.s. for every $t\ge 0$, where $\tilde \xi_t = \int_0^t \text{sgn}(X_s) d\xi_s$. In particular, in the case of reverse SLE, $\tilde \xi_t$ has the law of a Brownian motion and in particular, $(|\xi_s|, s \ge 0 )$ is stochastically dominated by $(\sqrt{\kappa} |B_s|, s \ge 0)$.
\end{remark}

We also note for further reference the following elementary upper bound on the height of a compact $\H$-hull in terms of its half-plane capacity:
\begin{lemma}
	Let $K$ be a compact $\mb H$-hull with $\height(K) > 2\sqrt \alpha$. Then $\hcap(K) \geq \alpha$.
	\label{lem:height.hcap}
\end{lemma}

\begin{proof}
An elementary argument (based on a reflection trick) is given in \cite{SLEnotes} which we include here for completeness. Since both the height of a hull and its half-plane capacity are translation invariant along the real line, and $\height(K) > 2 \sqrt \alpha$, we can assume that $2i\sqrt \alpha \in K$. Let $K'$ be the reflection of $K$ in the imaginary axis. Further, let $\tilde K$ be the complement of the connected component of $\mb H \setminus (K \cup K')$ which contains infinity. %(see Figure \ref{fig:k}).
%	\begin{figure}
%	\centering
%\def\svgwidth{\columnwidth}
%\input{k.pdf_tex}
%%\includegraphics{k.pdf_tex}
%\caption{Left: The hull $K$ and its reflection $K'$. Right: The joined %and filled $\tilde K$.}
%\label{fig:k}
%	\end{figure}
By construction, $\tilde K$ contains the line segment $(0, 2i\sqrt{\alpha} ]$ and therefore we know that
$
\hcap(\tilde K) \geq 2 \alpha.
$
Now recall that
$$\hcap(K) = \lim_{y \to \infty} y \E^{iy} (\Im (B_{\tau_K}))$$ where under $\P^{iy}$, $B$ is a Brownian motion started from $iy$ and $\tau_K$ is its first hitting time of $\partial H$.

%%%%%%%%%%
\begin{comment}
We can ignore the event that the Brownian motion exits through the real line since that does not contribute to the expectation, i.e.
\begin{equation*}
	\mb E_{iy}[\Im(B_{\tau_{\tilde K}})] = \mb E_{iy}\left[\Im(B_{\tau_{\tilde K}})\ind \{ B_{\tau_{\tilde K}} \in \partial \tilde K \}\right].
\end{equation*}
Since $\partial \tilde K \subset \partial K \cup \partial K'$, we know that
$
	\left\{ B_{\tau_{\tilde K}} \in \partial \tilde K \right\} \subset \left\{ B_{\tau_{\tilde K}} \in \partial K \right\} \cup \left\{ B_{\tau_{\tilde K}} \in \partial K' \right\}.
$
Therefore, we can bound the half-plane capacity of $\tilde K$ as follows:
\begin{equation}
	\mb E_{iy}\left[\Im(B_{\tau_{\tilde K}})\ind \{ B_{\tau_{\tilde K}} \in \partial \tilde K \}\right] \leq \mb E_{iy}\left[\Im(B_{\tau_{\tilde K}})\ind \{ B_{\tau_{\tilde K}} \in \partial K \}\right] + \mb E_{iy}\left[\Im(B_{\tau_{\tilde K}})\ind \{ B_{\tau_{\tilde K}} \in \partial K' \}\right]
	\label{eq:K}
\end{equation}
\end{comment}
%%%%%%%%%

Note that by symmetry, and since $\partial \tilde K \subset \partial K \cup \partial K'$,
%we know that
%\begin{equation*}
%	\mb E_{iy}\left[\Im(B_{\tau_{\tilde K}})\ind \{ B_{\tau_{\tilde K}} \in \partial K \}\right] = \mb E_{iy}\left[\Im(B_{\tau_{\tilde K}})\ind \{ B_{\tau_{\tilde K}} \in \partial K' \}\right],
%\end{equation*}
%and so, substituting this into \eqref{eq:K} we find
\begin{equation}
	\mb E_{iy}\left[\Im(B_{\tau_{\tilde K}})\ind \{ B_{\tau_{\tilde K}} \in \partial \tilde K \}\right] \leq 2\mb E_{iy}\left[\Im(B_{\tau_{\tilde K}})\ind \{ B_{\tau_{\tilde K}} \in \partial K \}\right].
	\label{eq:2k}
\end{equation}
Observe further that if $B_{\tau_{\tilde K}} \in \partial K$ then also $B_{\tau_{\tilde K}} = B_{\tau_K} \in \partial K$. Hence we get $\hcap(\tilde K) \leq 2 \hcap(K)$, and in turn
$
%\begin{equation*}
	\hcap(K) \geq \alpha,
%\end{equation*}
$
as desired.
\end{proof}

\subsection{Dirichlet GFF}

Before introducing the Gaussian free field (GFF) with free (or Neumann) boundary conditions, we first recall a few basic facts about the GFF with zero (or Dirichlet) boundary conditions. This is by now a classical object, and there are several introductions where standard properties can be found, see e.g., Chapter 1 of \cite{LQGnotes}.

 Let $C^\infty(D)$ be the set of smooth functions on $D$ with compact support.  The set $C^\infty(D)$ is made into a locally convex topological
vector space in which convergence is characterised as follows. A sequence $f_n \to 0$ in $C^\infty(D)$
if and only if there is a compact set $K \subset D$ such that $\text{Supp} f_n \subset  K$ for all $n$ and $f_n$ and all
its derivatives converge to 0 uniformly.

A distribution on $D$ is a continuous linear map $T : C^\infty(D) \to \R$. The set of distributions on $D$ is denoted $\cD'(D)$ and is given the
weak-$\star$ topology. Thus, $T_n \to T$ in $\cD'(D)$ if and only if $(T_n, f) \to (T,f)$ for all $f \in C^\infty(D)$.
The Borel $\sigma$-algebra on $\cD'(D)$ then coincides with that generated by the evaluation maps $T \mapsto (T,f)$ for $f \in C^\infty(D)$.

\begin{defn}[Dirichlet Inner Product]\label{def:dipd}
Let $D \subset \mb C$ be a proper simply connected domain. Let  $C^\infty_0(D)$ be the set of smooth, compactly supported functions on $D$. For $f, g$ in $C^\infty_0(D)$, define the Dirichlet inner product as
\begin{equation*}
	\dipd{f}{g} = \frac{1}{2\pi}\int_{D}\grad f(z) \cdot \grad g(z) dz.
\end{equation*}
This defines an inner product on $C_0^\infty(D)$, whose completion is by definition the Soboloev space $H_0^1(D)$.
\end{defn}

More generally, the Sobolev space $H^s(D)$ of index $s \in \R$ is defined as follows. Let $(e_n)_{n \ge 1}$ be an orthonormal basis of $L^2(D)$ (with respect to the $L^2$ inner product), and let $\{\lambda_n\}_{n \ge 1}$ be the eigenvalues of $- \Delta$ with Dirichlet boundary conditions ordered so that $0 < \lambda_1 < \lambda_2\le  \ldots $. We define $H^s(D)$ to be the completion of $C^\infty(D)$ with respect to the scalar product $(f,g)_s = \sum_{n=1}^\infty (f,e_n)^2 \lambda_n^s$.

The point of view which we adopt for defining the Dirichlet GFF is the following series expansion in $H_0^1(D)$.

\begin{defn}[Zero boundary Gaussian free field]
	Let $\{ f_i \}$ be an orthonormal basis of $H_0^1(D)$ with respect to the Dirichlet inner product. The GFF with Dirichlet boundary conditions is the series
\begin{equation}\label{eq:gffzero}
	 h = \sum_{i = 1}^\infty X_i f_i,
\end{equation}
where the $\{X_i\}$ are i.i.d.~standard Normal random variables.
\end{defn}

An eigenvalue calculation (more precisely, Weyl's law for Dirichlet eigenvalues) shows that when $D$ is bounded, if $\lambda_n$ are these eigenvalues, then $\sum_n X_n^2 \lambda_n^{-1+s}< \infty$ as soon as $s<0$. Since $e_n = \sqrt{\lambda_n} f_n$ forms an orthonormal basis of $L^2(D)$ (using the Gauss Green formula, or integration by parts), we deduce that the sum \eqref{eq:gffzero} converges a.s. in $H^s(D)$ for any negative index $s<0$, and hence in the space of distributions $\cD'(D)$. An argument of conformal invariance shows that this remains true for arbitrary proper simply connected domains of $\C$. See \cite{LQGnotes} for details. Alternatively, $h$ can be defined as a stochastic process indexed by signed measures $\rho = \rho_+ - \rho_-$ such that $\iint G_D(x,y) \rho_\pm(dx) \rho_\pm(dy) < \infty$, where $G_D$ denote the Green function (with Dirichlet boundary conditions) on $D$, such that $(h, \rho)$ is a centered Gaussian random variable with variance $\var (h, \rho) = \iint G_D(x,y) \rho(dx) \rho(dy)$. In that case there exists a modification of $h$ which is also a distribution on $D$ almost surely and which coincides with \eqref{eq:gffzero}. Again, we refer to \cite{LQGnotes} for details.

\subsection{Neumann GFF}\label{subs:fbgff}

%In order to couple the reverse SLE of Section \ref{subsec:reverse} with a Gaussian free field, as in \cite{sheffield2010conformal}, we need to understand the Neumann (or free) boundary Gaussian free field.

We now discuss the Neumann GFF case, which has a few additional complications due to the fact that it is only defined up to a global additive constant.

% We equip $\bar C^\infty(D)$ with the quotient topology, thus $\bar f_n \to 0$ if there is a compact set $K \subset \bar D$ and representatives $f_n$ of $\bar f_n$ such that $\text{Supp} f_n \subset K$ with $f_n \to 0$ uniformly as well as all its derivatives.

 We introduce the equivalence relation $T_1\sim T_2$ on distributions $T_1, T_2 \in \cD'(D)$ if $T_1 - T_2$ is the constant distribution $\lambda$ for some $\lambda \in \R$ (i.e., $(T_1, f) = (T_2, f) + \lambda \int f$ for all $f \in C^\infty(D)$ ). We call the quotient space $\bar \cD'(D) = \cD'(D)/ \sim$ the space of \textbf{distribution modulo constants}.

 Equivalently, let $\tilde C^\infty(D)$ be the set of smooth functions $f \in C^\infty(D)$ with mean zero, i.e., having the property that $\int_D f = 0$. Then $\tilde C^\infty(D)$ is also a locally convex topological vector space with the topology of $C^\infty(D)$, and it is easy to see that $\bar \cD'(D)$ can be identified with continuous linear forms over $\tilde C^\infty(D)$. Indeed, if $\bar T \in \bar \cD'(D)$ and $f \in \tilde C^\infty(D)$ then $(\bar T, f)$ is unambiguously defined. This allows us to equip $\cD'(D)$ with the quotient topology of $\cD'(D)$ or, equivalently, with the weak-$\star$ topology inherited from $\tilde C^\infty(D)$. Thus $\bar T_n \to \bar T$ if and only for all $f \in \tilde C^\infty(D)$, we have $(\bar T_n,  f_n) \to (\bar T, f)$.

 Introduce on $C^\infty(D)$ the equivalence relation $f\sim g$ if and only if $f-g$ is a constant. Let $\bar C^\infty(D) = C^\infty(D) / \sim$ be the quotient space.
 Note that $\langle \cdot, \cdot \rangle_\nabla$ still defines an inner product on $\bar C^\infty(D)$ (and also on $\tilde C^\infty(D)$ but this won't be relevant here). We let $\bar H(D)$ be the Hilbert space completion of $\bar C^\infty(D)$ with respect to $\| \cdot \|_\nabla$. Note that if $D$ has a smooth boundary and if $\bar N(D)$ denote the set of smooth functions (up to and including the boundary, with bounded support) which have Neumann boundary conditions, i.e., $\nabla f \cdot n = 0$ where $n$ is a normal vector to $\partial D$, then $\bar N(D)$ is dense in $\bar C^\infty(D)$ with respect to $\| \cdot \|_\nabla$, hence the Hilbert space completion of $\bar N(D)$ is the same as that of $\bar C^\infty(D)$ and thus also $\bar H(D)$.

\begin{comment}
The Dirichlet inner product defined in Definition \ref{def:dipd} can be extended from the set of smooth, compactly supported functions to $f,g \in C^\infty({D})$, the set of smooth functions on ${D}$. It is no longer strictly an inner product, as all constant functions will have Dirichlet inner product equal to zero. To get around this, we define the space $\overline H(D)$ as the Hilbert space completion under the Dirichlet inner product of the \emph{subset} of functions $f \in C^\infty({D})$ which have the property that
\begin{equation*}
	\int_{D} f(z)dz = 0.
\end{equation*}
\end{comment}

\begin{defn}[Neumann boundary Gaussian free field]
Let $\{ \bar f_i \}_{i\ge 0}$ be an orthonormal basis of $\bar H(D)$ under the Dirichlet inner product. The Neumann boundary Gaussian free field is the series
\begin{equation}\label{eq:gff}
	h = \sum_{i = 0}^\infty X_i \bar f_i,
\end{equation}
where the $\{X_i\}$ are i.i.d.~standard Normal random variables.
\end{defn}

It is not immediately clear, but true, that the series \eqref{eq:gff} converges in the space of distribution-modulo-constants.
%(equivalently, linear forms on functions $f \in C^\infty(D)$ having the property that $\int_D f(z) dz = 0$).
When $D$ is smooth and bounded, this follows from the fact that the Weyl law also holds under these assumptions for Neumann eigenvalues $\{\mu_n\}_{n \ge 0}$ (where $0 = \mu_0 < \mu_1 \le \mu_2\ldots $); see Courant and Hilbert, \cite[VI.4, Theorem 16]{CourantHilbert}. Hence we deduce that $\sum_{n} X_n^2  \mu_n^{-1+s} < \infty$ a.s for all $s<0$, just as in the Dirichlet case. Hence if we take $(f_n)_{n \ge 0}$ to be an orthonormal basis of Neumann eigenfunctions then $e_n = \sqrt{\mu_n} f_n$ is also an orthonormal basis of $L^2(D)$ (again by the Gauss Green formula) hence, since $\lambda_n \sim \mu_n$ as $n \to \infty$ (indeed both families satisfy Weyl's law), the series $\sum_n X_n f_n$ (which is a particular representative of $h$) converges in $H^s(D)$. This in turn implies convergence in the space of distributions modulo constants. (See \cite{LQGnotes} for details.)
A conformal invariance argument implies, just as in the Dirichlet case, that the series \eqref{eq:gff} converges in $\bar \cD'(D)$ for arbitrary simply connected domains, since orthonormal basis of $\bar H(D)$ are conformally invariant.

\medskip We deduce the following lemma, which allows us to talk about the harmonic extension of the boundary data of $h$.

%\begin{equation*}
%	\mb E \left[ \dipd h f \dipd h g \right] = \dipd f g.
%\end{equation*}
%
%Note that, since we only ever look at test functions with zero integral, the Neumann boundary GFF is defined as a distribution only up to an additive constant. Later, we will need to fix this additive constant.

%We will often talk about the ``harmonic extension'' of the Neumann boundary GFF from the boundary $\partial D$ to the domain $D$. The following lemma allows us to do that.

\begin{lemma}\label{L:harmext}
	Let $h$ be a Neumann GFF on a domain $D$. Then we can write
	\begin{equation}\label{harmext}
		h = \tilde h + \harmd (h),
	\end{equation}
	where $\tilde h$ is a zero boundary GFF on $D$ and $\harmd (h)$ is an independent harmonic function on $D$, defined up to an additive constant. We call $\harmd (h)$ the \textbf{harmonic extension of $h$} in $D$ from $\partial D$.
	\label{lem:decomp}
\end{lemma}
\begin{proof}
	Let $\overline{\text{Harm}}(D)$ denote the set of harmonic functions on $D$, up to constants. The proof comes from the fact that the space $H_0^1(D)$ and the space of harmonic functions defined up to additive constants on $D$, $\overline{\text{Harm}}(D)$, are orthogonal complementary subsets of $\bar H(D)$, which can be proved exacty in the same way as the standard Markov property of the Dirichlet GFF (see e.g. Theorem 1.17 in \cite{LQGnotes}). We can therefore find an orthonormal basis $\{\phi_i\}$ of $H^1_0(D)$, and an orthonormal basis $\{\bar \psi_i\}$ of $\overline{\text{Harm}}(D)$, such that together they form an orthonormal basis of $\bar H(D)$. We can then decompose the Neumann boundary GFF, $h$, in terms of this basis:
	\begin{equation*}
		h = \sum_{i = 1}^\infty X_i \phi_i + \sum_{i = 1}^\infty Y_i \bar \psi_i,
	\end{equation*}
	where $\{X_i\}$ and $\{Y_i\}$ are i.i.d.~standard Normal random variables. The first sum is exactly the one in \eqref{eq:gffzero}, i.e.~it defines a zero boundary GFF, and hence that series converges a.s. as a distribution.
	Since the sum of the two series must converge as a distribution modulo constants by the above, and there is a representative for which the series converges as a distribution (as discussed above), the second series $\sum_i Y_i \bar \psi_i$ must also converge as a distribution modulo constants and we can find representative $\psi_i$ such that the sum $\sum_{i = 1}^\infty Y_i \psi_i$ converges as a distribution. Hence the infinite sum is harmonic as a distribution. Standard elliptic regularity arguments then imply that it is then a function which is harmonic in the classical sense; hence $\sum Y_i \bar \psi_i \in \overline{\text{Harm}}(D)$.
\end{proof}

Occasionally it will be convenient to fix the free additive constant of $h$, i.e., to pick a particular representative of $h$. Equivalently, this amounts to choosing a representative for $\harmd(h)$ in the decomposition \eqref{harmext}, since we will always insist that $\tilde h$ has zero boundary conditions. We then say that we fix a \textbf{normalisation} for $h$. With an abuse of notation, we will still call $h$ the chosen representative, and $\harmd(h)$ the corresponding harmonic part.

\begin{remark}
 It is important to note that when we fix a normalisation of $h$, it can be the case that $\tilde h$ and $\harmd(h)$ in the decomposition \eqref{harmext} are no longer independent. (For example, if we normalise the GFF $h$ by specifying that its average value on some set is equal to zero, there will be some interaction between $\tilde h$ and $\harmd(h)$.)
 %This is because the zero boundary GFF $\tilde h$ is fully specified, so the function $\harmd(h)$ will have to shift by a constant to compensate for the average value of $\tilde h$ on the set in question.
\end{remark}

We will always normalise by specifying the value of $\harmd(h)$ at a point. This will ensure that we still have independence between the harmonic and zero boundary condition parts, despite the above remark:

\begin{lemma}
	Let $h$ be a Neumann boundary GFF on a domain $D$, and let $z_0 \in D$. If we normalise $h$ so that $\harmd(h)(z_0) = 0$ in the decomposition of Lemma \ref{lem:decomp}, then $\tilde h$ and the corresponding harmonic part $\harmd(h)$ are independent.
	\label{lem:independence}
\end{lemma}

\begin{proof}
This is immediate from \cref{L:harmext} since if we write $h = \tilde h + \bar u$ where $\tilde h$ is a Dirichlet GFF and $\bar u$ is an independent harmonic function up to constants, then
the unique representative $u$ of $\bar u$ such that $u(z_0) = 0$ depends only on $\bar u$ (i.e., is a measurable function of $\bar u$) and is thus still independent of $\tilde h$.
%	In Lemma \ref{lem:decomp}, $\harmd (h) \in \overline{\text{Harm}}(D)$ represents an equivalence class of harmonic functions, $[\harmd(h)] \subset \text{Harm}(D)$, where two functions are equivalent if they agree up to an additive constant. We can choose the unique representative of that equivalence class which equals $0$ at $z$, and call that function $\harmd(h)$. Since $\harmd(h)$ depends only on the equivalence class $[\harmd(h)]$, and the equivalence class is independent of $\tilde h$, then $\harmd(h)$ is also independent of $\tilde h$.
\end{proof}

We will need some quantitative bounds on the variance of the harmonic part of the Neumann boundary GFF when it is pinned at a certain point. We will use the following, which is easy to deduce from \cref{L:harmext} as well as the expressions for the Green function of the Neumann GFF and the Dirichlet GFF (see (5.8) in \cite{LQGnotes}, and see also  Lemma 2.9 in \cite{GwynneMillerSun} for a different proof of this result):
\begin{lemma}
	Let $\harmdi(h)$ be the harmonic part of a Neumann boundary GFF on the unit disc $\mb D$, normalised so that $\harmdi(h)(0) = 0$. Then, for any $z, w \in \mb D$, $\harmdi(h)(z)$ and $\harmdi(h)(w)$ are jointly Gaussian with mean zero and covariance
	\begin{equation*}
		\mb E \left[ \harmdi(h)(z) \harmdi(h)(w) \right] = - 2\log |1 - z \overline w|.
	\end{equation*}
	\label{lem:disc.cov}
\end{lemma}

We can use a coordinate change from the upper half plane to the unit disc, along with conformal invariance of Gaussian free field, to get the following bound on the variance of the harmonic part of the Neumann boundary GFF on the upper half plane.

\begin{lemma}
	Let $\harm(h)$ be the harmonic part of a Neumann boundary GFF on the upper half plane $\mb H$, normalised so that $\harm(h)(iy_0) = 0$ for some (fixed) $y_0 > 0$. Then, for any $z = x + iy \in \mb H$,
	\begin{equation*}
		\mb E \left[ \harm(h)(z)^2 \right] =-2 \log  \frac{4y_0y}{x^2 + (y+y_0)^2}.
% \leq -2 \log y + 2\log\left( x^2 + 4y_0^2 \right) - 2\log 4y_0.
	\end{equation*}
	\label{lem:half.plane.cov}
\end{lemma}
\begin{proof}
	Let $h'$ be a Neumann boundary GFF on the unit disc, $\mb D$, normalised so that
%	\begin{equation*}
$
		\harmdi(h')(0) = 0,
$
%	\end{equation*}
	and let $m_{y_0}$ be the M\"obius transformation
	\begin{equation*}
		m_{y_0}(z) = \frac{z - iy_0}{z + iy_0}.
	\end{equation*}
	Then $m_{y_0}:\mb H \to \mb D$ so that $m_{y_0}(iy_0) = 0$. Therefore, if we set
	%\begin{equation*}
	$
	h = h' \circ m_{y_0},
    $
	%\end{equation*}
	then $h$ is a Neumann boundary GFF on $\mb H$ and, by conformal invariance of the GFF and harmonic extensions, we see that for $z \in \mb H$,
	\begin{align*}
		\harm(h)(z) &= \harm(h' \circ m_{y_0})(z) = \harmdi(h')(m_{y_0}(z)).
	\end{align*}
	Therefore, with $z = iy_0$, we see that $\harm(h)(iy_0) = \harmdi(h')(m_{y_0}(iy_0)) = 0$. So, $h$ is normalised in the way that we want. We can now calculate its variance using the coordinate change and Lemma \ref{lem:disc.cov}, as follows. Let $z = x + iy \in \mb H$ with $y \in (0, y_0)$. By Lemma \ref{lem:disc.cov},
	\begin{equation}
		\mb E \left[ \harm(h)(x + iy)^2 \right] = \mb E \left[ \harmdi(h')(m_{y_0}(x + iy))^2 \right] = -2 \log| 1 - |m_{y_0}(x+iy)|^2 |. \label{eq:var}
	\end{equation}
	Expanding the term inside the $\log$ in \eqref{eq:var}, we find
	\begin{align}
		1 - |m_{y_0}(x + iy)|^2 = 1 - \frac{x^2 + (y-y_0)^2}{x^2 + (y+y_0)^2}
%
%		\nonumber &= \frac{(y+y_0)^2 - (y-y_0)^2}{x^2 + (y+y_0)^2} \\
		= \frac{4y_0y}{x^2 + (y+y_0)^2}.
		\label{eq:ma}
	\end{align}
	Substituting \eqref{eq:ma} into \eqref{eq:var}, we obtain the desired result.
%	\begin{align*}
%		\mb E \left[ \harm(h)(x + iy)^2 \right] &= -2 \log \left( \frac{4y_0y}{x^2 + (y+y_0)^2} \right) %\\
%		&= -2 \log y + 2\log\left( x^2 + (y + y_0)^2 \right) - 2\log 4y_0
%\\
%		&\leq -2 \log y + 2 \log(x^2+4y_0^2) - 2\log 4y_0,
%	\end{align*}
%as desired.
%	where we have used the assumption that $y < y_0$ to get the last inequality.
\end{proof}

\begin{corollary}\label{cor:imag.cov}
	Let $\harm(h)$ be the harmonic part of a Neumann boundary GFF on the upper half plane $\mb H$, normalised so that $\harm(h)(iy_0) = 0$ for some (fixed) $y_0 > 0$. Then, for any purely imaginary $z = iy \in \mb H$ with $y \in (0, y_0)$, we see that
	\begin{equation*}
		\mb E \left[ \harm(h)(iy)^2 \right] \leq 2\log y_0 - 2 \log y.
	\end{equation*}
\end{corollary}

\begin{proof}
	Note that if $y< y_0$ in \cref{lem:half.plane.cov} then $ \E (\harm(h)(z)^2) \leq -2 \log y - 2\log 4y_0 + 2\log\left( x^2 + 4y_0^2 \right) .$
In particular setting $x = 0$ gives the result.
\end{proof}

	Let $K$ be a compact hull, and let $H = \mb H \setminus K$. Let $h$ be a Neumann boundary GFF on $\mb H$ with some normalisation.
% normalised so that its harmonic part vanishes at $iy_0$ for some $y_0 > 0$.
By \cref{L:harmext} we can consider write $h = \tilde h + u$ where $u$ is a harmonic function, and $\tilde h$ is a Dirichlet GFF. Applying the domain Markov property to $\tilde h$ inside $H$, we can write $\tilde h|_H = h'+ v$ where $v$ is harmonic in $H$ and $h'$ is a Dirichlet GFF on $H$. We then call $v +u =: \harmh(h)$ the harmonic extension of $h$ from $\partial H$ to $H$, so that
$$
h|_H = h' + \harmh(h),
$$
where $h'$ is a Dirichlet GFF on $H'$ and note that $h'$ is independent from $\harmh(h)$ viewed as a function up to constant. If we normalise $h$ so that $\harm(h) (z) = 0$ at some $z$, then $\harmh(h)$ (now viewed as a proper harmonic function) is independent of $h'$.

\begin{lemma}
Let $h, K$ be as above and assume that $h$ is normalised so that its harmonic part vanishes at $iy_0$ for some fixed $y_0>0$. Then for any $z = x + iy \in H$ with $y < y_0$, letting $\harmh(h)$ be the corresponding harmonic function,
%we know that $\harmh(h)(z)$ is a centred Gaussian random variable with variance
	\begin{equation*}
		\mb E \left[ \harmh(h)(z)^2 \right] \leq -3 \log (\dist(z, \partial H)) + 2\log(x^2 + 4y_0^2) + C,
	\end{equation*}
	where $C$ is a constant which depends on $y_0$ only.
	\label{lem:maybe}
\end{lemma}

\begin{proof}
%%%%%%%%%%%%%%
\begin{comment}
	We know from Lemmas \ref{lem:decomp} and \ref{lem:independence} that we can write $h = \tilde h +  \harm(h)$, where $\tilde h$ is a zero boundary GFF on $\mb H$ and $\harm(h)$ is an independent harmonic function, determined uniquely. Using the Markov property of the zero boundary GFF we can also say that, restricting to $H$, we have
	\begin{equation*}
		\tilde h |_H = h' + u,
	\end{equation*}
	where $h'$ is a zero boundary GFF on $H$ and $u$ is an independent harmonic function. We can therefore see that
	\begin{align*}
		h|_H &= h' + u + \harm(h)|_H \\
		&= h' + \harmh(h).
	\end{align*}
	Because $u$ is independent of $h'$ and $\harm(h)$ is independent of $\tilde h = h' + u$, we see that $h'$ and $\harmh(h) = u + \harm(h)$ are independent.
	\end{comment}
%%%%%%%%%%%%%%%%%
Write $h|_H = h'+ \harmh(h)$ as above, where the two terms are independent and $h'$ is a Dirichlet GFF on $H$, and we also write $h = \tilde h + \harm(h)$.
		Let $\varepsilon < \dist(z, \partial H)$. Then we can take the circle average of $h$ in both expressions to see that
	\begin{equation*}
		h'_\varepsilon(z) + \harmh(h)(z) = \tilde h_\varepsilon(z) + \harm(h)(z).
	\end{equation*}
	By independence of both terms on the left hand side as well as on the right hand side, taking the variance, we have:
	\begin{equation}
		\var(\harmh(h)(z)) \leq \var(\tilde h_\varepsilon(z)) + \var(\harm(h)(z)).
		\label{eq:varvarvar}
	\end{equation}
	Now it is well known (see Lemma 2.2 in \cite{LQGnotes}) that $\var (\tilde h_\eps (z)) = - \log \eps + \log R(z; \H)$ where $R(z;D)$ denotes the conformal radius of a point $z$ in a simply connected domain $D$.
Using \ref{lem:half.plane.cov} and the fact that $y < y_0$ we obtain
%to bound the terms on the right hand side of \eqref{eq:varvarvar} to get
	\begin{equation*}
		\var(\harmh(h)(z)) \leq -\log \varepsilon + \log R(z; \mb H) -2\log y + 2 \log(x^2 + 4y_0^2) - 2\log 4y_0.
	\end{equation*}
Clearly, $\log R(z; \mb H)$ is uniformly bounded above for all $z = x + iy$ with $y \in [0, y_0]$. Let us write this upper bound along with the $-2\log 4y_0$ term as the constant $C$, so
	\begin{equation*}
		\var(\harmh(h)(z)) \leq -\log \varepsilon -2\log y + 2 \log(x^2 + 4y_0^2) + C.
	\end{equation*}
	Now let us take $\varepsilon$ as large as we can, i.e.~$\varepsilon = \dist(z, \partial H)$. Noting also that, as $H$ is a subset of $\mb H$, $\dist(z, \partial \mb H) = y \geq \dist(z, \partial H)$, we find
	\begin{equation*}
		\var(\harmh(h)(z)) \leq -3\log \dist(z, \partial H) + 2 \log(x^2 + 4y_0^2) + C,
	\end{equation*}
	as desired.
\end{proof}

\subsection{Coupling reverse SLE and Neumann boundary GFF}\label{subsec:coupling}

We now set out the coupling between reverse SLE and the Neumann boundary GFF, discovered by Sheffield in \cite{zipper}. Throughout, fix $\kappa > 0$ and $Q = \frac{2}{\sqrt \kappa} + \frac{\sqrt \kappa}{2}$.

\begin{thm}[Liouville quantum gravity coupling, or reverse coupling]\label{thm:coupling}
Let $(f_t)$ be the reverse SLE$_\kappa$ flow as defined by \eqref{eq:reverse}.
Let $h$ be a Neumann boundary GFF on $\mb H$, independent of $(f_t)$. %normalised so that the harmonic part vanishes at $iy_0$ for some $y_0 > 0$ to be specified later.
For $t > 0$, let
\begin{equation}
	h_t = h \circ f_t + \frac{2}{\sqrt \kappa}\log |f_t|.
	\label{eq:ht}
\end{equation}
Then
\begin{equation*}
	h_t + Q\log |f_t'| \overset d = h_0
\end{equation*}
where the equality is in distribution for the two sides of the identity, viewed as (Schwartz) distributions modulo constants.
\end{thm}

The proof is a fairly simple application of It\^o's formula, see Theorem 6.1 in \cite{LQGnotes}. (There the statement is interpreted as expressing a domain Markov property for certain random surfaces explored along an SLE interface.)

\begin{corollary}\label{cor:coord}
	In the same setting as Theorem \ref{thm:coupling}, there exists a constant (in space) $b_t$ such that if $h$ is normalised so that its harmonic part vanishes at $i y_0$ for some fixed $y_0>0$,
	\begin{equation*}
	|f_t'(iy)| = \left( \frac{y}{|f_t(iy)|} \right)^{\frac{2}{Q\sqrt \kappa}}\exp\left( \frac{1}{Q}\left( \harm(h')(iy) - \harm(h\circ f_t)(iy) + b_t \right) \right),
	\end{equation*}
	where $h' \overset d = h$ (viewed as distributions with a normalisation).
%is also a Neumann boundary GFF, normalised so that the harmonic part vanishes at $iy_0$.
\end{corollary}
\begin{proof}
		Theorem \ref{thm:coupling} tells us that $h_t + Q\log |f_t'| \overset d = h_0$ modulo an additive constant. Hence if we fix the normalisation of $h$ so that its harmonic part vanishes a $i y_0$, we can write
\begin{equation}
	h_t + Q\log|f_t'| - b_t \overset d = h_0
	\label{eq:dist}
\end{equation}
for some constant $b_t$, where $h_t = h \circ f_t +  (2/\sqrt{\kappa}) \log |f_t|$ as in \eqref{eq:ht}, and here the equality is as distributions. Let
\begin{equation}
	h_0' = h_t + Q\log|f_t'| - b_t
	\label{eq:hprime}
\end{equation}
so that $h_0' \overset d = h_0$, that is, $h_0' = h ' + \frac{2}{\sqrt \kappa}\log|z|$ where $h'$ is a Neumann boundary GFF on $\mb H$ normalised so that its harmonic part vanishes at $i y_0$.

Rearranging \eqref{eq:hprime} gives us for $z \in \H$,
\begin{align}
	\nonumber Q\log |f_t'(z)| &= (h_0' - h_t)(z) + b_t \\
	\nonumber&= (h' - h\circ f_t)(z) + \frac{2}{\sqrt \kappa}\log|z| - \frac{2}{\sqrt \kappa} \log |f_t(z)| + b_t \\
	&= (h' - h\circ f_t)(z) + \frac{2}{\sqrt \kappa}\log \frac{|z|}{|f_t(z)|} + b_t.
	\label{eq:coordchange}
\end{align}
The left hand side of \eqref{eq:coordchange} is harmonic in $\mb H$, as are the last two terms on the right hand side. The terms $h'$ and $h \circ f_t$ are not defined pointwise, but their difference is. Furthermore, since all other terms are harmonic, the difference $h' - h\circ f_t$ must also be harmonic.

Integrating with respect to the Poisson kernel on the real line lets us look at the harmonic extensions of these function from $\mb R$ to $\mb H$. As both sides of the equation are harmonic already, this has no effect on the functions. However, linearity of integration against the Poisson kernel lets us say that
\begin{equation}
	Q\log |f_t'(z)| = \harm(h')(z) - \harm(h\circ f_t)(z) + \frac{2}{\sqrt \kappa}\log \frac{|z|}{|f_t(z)|} + b_t,
	\label{eq:harmcoordchange}
\end{equation}
so that taking the exponential, we get
\begin{equation*}
	|f_t'(iy)| = \left( \frac{y}{|f_t(iy)|} \right)^{\frac{2}{Q\sqrt \kappa}}\exp\left( \frac{1}{Q}\left( \harm(h')(iy) - \harm(h\circ f_t)(iy) + b_t \right) \right),
	\label{eq:ft}
\end{equation*}
as desired.
\end{proof}

\section{Proof of Theorem \ref{T:mainq}}\label{sec:proof}

\begin{comment}
We restate the theorem:
\begin{thm}
	Let $\kappa \neq 8$, and let $\hat f_t $ be the centred inverse of the Loewner flow. Then there exist constants $\varepsilon > 0$, $\delta > 0$ and $C > 0$ such that
	\begin{equation}
		\mb P \left[ |\hat f'_t(iy)| > y^{-(1-\varepsilon)} \right] \leq C y^{2 + \delta}
		\label{rs:eq:pbound}
	\end{equation}
	for all $t \in [0,1]$ and $y \in (0,1)$.
	\label{rs:thm:tail2}
\end{thm}
\end{comment}

%Thanks to the distributional equality between the reverse flow, $f_t$, and the centered inverse, $\hat f_t$ shown in
By Lemma \ref{lem:relations}, it suffices to prove the following:

\begin{prop}\label{prop:prob1}
	Let $(f_t)$ be a reverse $SLE_\kappa$ for $\kappa > 0$ and $\kappa \neq 8$, coupled with a Neumann boundary GFF as in Corollary \ref{cor:coord} and let $(b_t)_{t\geq0}$ be the coupling constants.
%Finally, let $(\xi_t)_{t \geq 0}$ be the driving function of the reverse SLE $f$.
Then for all $\delta > 0$, there exist $\eps > 0$ and $C > 0$ such that, for all $t \in [0,1]$ and $y \in [0,1]$
	\begin{equation*}
		\mb P \left[ |f_t'(iy)| > y^{-(1-\varepsilon)} \right] \leq Cy^{q-\delta},
	\end{equation*}
where as before, $q = 4/\kappa + \kappa/16 + 1$.
\end{prop}

\begin{proof}
	Let $(\xi_t)_{t\ge 0}$ be the driving function of the reverse SLE$_\kappa$ flow. Let us fix $\eps$ for now. Then note that
%We need to break the probability up into events which are relatively easy to deal with. A simple union bound lets us say that
	\begin{align}
		\mb P \left[ |f_t'(iy)| > y^{-(1-\varepsilon)} \right] &\leq \mb P \left[ |f_t'(iy)| > y^{-(1-\varepsilon)}, \: b_t \leq -\varepsilon \log y, \: \sup_{t \in [0,1]} |\xi_t | \leq y^{-\varepsilon} \right] +  \label{mainterm}\\
		& \hspace{2em} + \mb P\left[ b_t > -\varepsilon \log y \right] + \mb P\left[ \sup_{t \in [0,1]} |\xi_t| > y^{-\varepsilon} \label{minorterm}\right].
	\end{align}
The bulk of the work consists in showing that the first term \eqref{mainterm} decays as $C y^{q- \delta}$, which is carried in Section \ref{subsec:prob}, Proposition \ref{prop:prob} in particular.

We separately show that the coupling constant $b_t$ has sub-exponentially decaying tail in Section \ref{sec:coupling.constant}, giving us the arbitrary polynomial decay that we need here. Finally, the reflection principle implies that the supremum and the infimum of a Brownian motion over a finite time interval both have Gaussian and hence sub-exponential tails, and so the third term decays faster than any polynomial as $y \to 0$. These two arguments therefore imply that \eqref{minorterm} decays faster than any polynomial as $y \to 0$.
\end{proof}

\subsection{Bounding the coupling constant}\label{sec:coupling.constant}
The aim of this section is to prove the upper bound that we need for the coupling constant $b_t$ that was introduced in Corollary \ref{cor:coord}. We want to show that it has the following polynomial upper tail:
\begin{prop}\label{lem:bt}
	Suppose $y_0 > 2\sqrt{2}$. Then the constant $b_t$ from Corollary \ref{cor:coord} has sub-exponential decay i.e.~for any $\lambda > 0$ there exists some constant $C_\lambda$ depending only on $\lambda$ (and on the choice of $y_0 > 2\sqrt{2}$) such that, for all $x > 0$ and all $0 \le t \le 1$, we have:
	\begin{equation*}
		\mb P \left[ b_t > x \right] \leq C_\lambda e^{-\lambda x}.
	\end{equation*}
	\end{prop}

\begin{proof}
	Let $\lambda > 0$. By Markov's inequality it suffices to check that $\mb E \left[ e^{\lambda b_t} \right]$ is finite for all $\lambda > 0$, with a uniform bound for $t\in[0,1]$. We can rearrange \eqref{eq:harmcoordchange} to see that, for $z \in \H$,
	\begin{equation*}
		b_t = \harm(h \circ f_t)(z) - \harm(h')(z) + Q\log|f'_t(z)| + \frac{2}{\sqrt \kappa}\log \frac{|f_t(z)|}{|z|}.
	\end{equation*}
Taking $z = iy_0$, the point where $\harm(h')$ vanishes (which is how $b_t$ is chosen), we can see that
	\begin{equation*}
		b_t = \harm(h \circ f_t)(iy_0) + Q\log|f'_t(iy_0)| + \frac{2}{\sqrt \kappa}\log \frac{|f_t(iy_0)|}{y_0}.
	\end{equation*}
	Exponentiating and recalling that $f_t$ is independent of $h$, we find that
	\begin{equation}
		\mb E \left[ e^{\lambda b_t} | f_t\right]
		= |f'_t(iy_0)|^{\lambda Q} \left( \frac{|f_t(iy_0)|}{y_0} \right)^{\frac{2 \lambda}{\sqrt \kappa}} \mb E \left[ \exp \left( \lambda \harm (h \circ f_t)(iy_0) \right) | f_t\right].
		\label{eq:ebt}
	\end{equation}
	We can use Lemma \ref{lem:maybe} to bound the conditional expectation $\mb E \left[ \exp \left( \lambda \harm (h \circ f_t)(iy_0) \right) | f_t \right]$, using the fact that $\harm(h \circ f_t)(iy_0) = \text{Harm}_{H_t}(h)(f_t(iy_0))$, as follows:
	\begin{align}
		\nonumber&\mb E \left[ \exp\left( \lambda \harm(h \circ f_t)(iy_0) \right)| f_t  \right] \leq \\
		&\hspace{2em}\leq\exp\left( \frac{\lambda^2}{2} \left( -3 \log(\dist(f_t(iy_0),\partial H_t)) + 2 \log(\Re(f_t(iy_0))^2 + 4y_0^2) + C \right) \right).
		\label{eq:harmexp}
	\end{align}
	Now we can use the fact that $\Im(f_t(iy_0))$ is nondecreasing and \cref{lem:height.hcap}, to deduce
	\begin{align}
		\nonumber \dist(f_t(iy_0), \partial H_t) &\geq \Im(f_t(iy_0)) - \height(K_t)
		\geq y_0 - 2\sqrt{2t} \geq y_0 - 2\sqrt 2,
		\label{eq:dist.bound}
	\end{align}
	and by choosing any $y_0 > 2 \sqrt{2}$ the right hand side positive.
	% \eqref{rs:eq:y0} ensuring that this is strictly positive.
	
	%Using the inequality \eqref{eq:dist.bound} in \eqref{eq:dist.bound} we get
	%\begin{equation}
	%	\mb E_h \left[ \exp\left( \lambda \harm(h \circ f_t)(iy_0) \right) \right] \leq (y_0 - 2\sqrt 2)^{-\frac{3\lambda^2}{2}}(\Re(f_t(iy_0))^2 + 4 y_0^2)^{\lambda^2} e^{\frac{\lambda^2 C}{2}}.
	%	\label{eq:harmexp2}
	%\end{equation}
	%Substituting \eqref{eq:harmexp2} into \eqref{eq:ebt} gives us the inequality
	Consequently, we obtain
	\begin{equation}
		\mb E\left[ e^{\lambda b_t}  | f_t \right] \leq |f_t'(iy_0)|^{\lambda Q} \left( \frac{|f_t(iy_0)|}{y_0} \right)^{\frac{2\lambda}{\sqrt \kappa}}(y_0 - 2\sqrt 2)^{-\frac{3\lambda^2}{2}}(\Re(f_t(iy_0))^2 + 4 y_0^2)^{\lambda^2} e^{\frac{\lambda^2 C}{2}}.
		\label{eq:btbound2}
	\end{equation}
	Corollary \ref{cor:fprime.bound} lets us bound the first term in \eqref{eq:btbound2} by
	\begin{equation*}
		|f_t'(iy_0)| \leq \frac{4}{y_0}\sqrt{y_0^2 + 4}.
	\end{equation*}
	Lemma \ref{lem:im.bound} lets us bound the second term in \eqref{eq:btbound2} by
	\begin{align*}
		|f_t(iy_0)| &\leq \sqrt{\Re(f_t(iy_0))^2 + 4t + y_0^2 }  \leq \sqrt{\Re(f_t(iy_0))^2 +  4y_0^2 },
	\end{align*}
	where the final inequality comes from the fact that we know $4t \leq 4 \leq 3y_0^2$.
	We obtain:
	\begin{equation*}
		\mb E\left[ e^{\lambda b_t} | f_t \right] \leq 4^{\lambda Q} y_0^{-\lambda Q - \frac{2\lambda}{\sqrt \kappa}} e^{\frac{\lambda^2 C}{2}} \frac{(y_0 + 4)^{\lambda Q/2}}{(y_0 - 2\sqrt 2)^{3\lambda^2/2}} \left( \Re(f_t(iy_0))^2 + 4y_0^2 \right)^{\lambda^2 + \frac{\lambda}{\sqrt \kappa}}.
	\end{equation*}
	We now take expectations and using \cref{lem:real.bound} ,
	\begin{align}
		\nonumber \mb E[e^{\lambda b_t}] &\leq  C( \lambda, y_0, \kappa)  \mb E \left[ \left( \Re(f_t(iy_0))^2 + 4y_0^2 \right)^{\lambda^2 + \frac{\lambda}{\sqrt \kappa}} \right] \\
		&\nonumber \leq C( \lambda, y_0, \kappa) \mb E \left[ \left( 4\kappa \sup_{s\le t}B_s^2 + 4y_0^2 \right)^{\lambda^2 + \frac{\lambda}{\sqrt \kappa}} \right] \\
		&\leq  C'( \lambda, y_0, \kappa),
		\label{eq:btboundfinal}
	\end{align}
	where the final line is an easy consequence of the reflection principle for Brownian motion. This finishes the proof of \cref{lem:bt}.
\end{proof}

\subsection{Bound on the main term \eqref{mainterm}}\label{subsec:prob}

We now deal with the main term \eqref{mainterm} in the proof of \cref{prop:prob1}.
%The good event that we bound the probability on introduces constraints on the supremum of the driving function of the SLE and a bound on the coupling constant $b_t$ introduced in Corollary \ref{cor:coord}. We show that these constraints both hold with very high probability in Section \ref{sec:coupling.constant}.
Throughout, we need to fix a point $i y_0$ for some $y_0 > 0$ that we use to normalise the Neumann boundary GFF used in the coupling arguments. In order to apply Lemma \ref{lem:maybe} we need to ensure that any complex point that we consider, especially those of the form $f_t(iy)$, have imaginary parts smaller than $y_0$. Happily, we need consider only times $t \in [0,1]$ and the starting points $iy$ with $y \in [0,1]$. Therefore, Lemma \ref{lem:im.bound} guarantees that
\begin{equation}
\label{L:bound_im}	\Im(f_t(iy)) \leq \sqrt{4t+iy} \leq \sqrt 5,
\end{equation}
for all $t\in[0,1]$ and $y\in[0,1]$. Since $2\sqrt{2} > \sqrt{5}$, it will suffice to assume that $y_0 > 2 \sqrt{2}$ for the arguments in this section and in \cref{lem:bt} to hold.

%Furthermore, it will be important in Section \ref{sec:coupling.constant} that
%\begin{equation*}
%y_0 > 2\sqrt2.
%\end{equation*}
%So we fix a point $iy_0$, with
%\begin{equation}
%	y_0 > \max\left( \sqrt 5, 2\sqrt 2\right) = 2\sqrt 2,
%	\label{rs:eq:y0}
%\end{equation}
%which we will use as the pinned point for the remainder of this chapter.

\begin{prop}\label{prop:prob}
%	Let $(f_t)$ be a reverse $SLE_\kappa$ for $\kappa > 0$ and $\kappa \neq 8$, coupled with a Neumann boundary GFF as in Corollary \ref{cor:coord} and let $(b_t)_{t\geq0}$ be the coupling constants. Finally, let $(\xi_t)_{t \geq 0}$ be the driving function of the reverse SLE $f$.
Consider the setup of \cref{prop:prob1}.
	Then for all $\varepsilon > 0$, there exists $\delta > 0$ and $C > 0$ such that, for all $t \in [0,1]$ and $y \in [0,1]$ we have
	\begin{equation*}
		\mb P \left[ |f_t'(iy)| > y^{-(1-\varepsilon)}, \: b_t \leq -\varepsilon \log y, \: \sup_{t \in [0,1]} |\xi_t| \leq y^{-\varepsilon} \right] \leq Cy^{q-\delta}.
	\end{equation*}
\end{prop}

We fix $\eps>0$, $ 0 \le t \le 1$, and $0 \le y \le 1$, and introduce the events
$$
\cA = \cA(t, y, \eps) =  \{ |f'_t(iy)| > y^{-(1-\varepsilon)} \}  \cap \{ \sup_{t \in [0,1]}|\xi_t| \leq y^{-\varepsilon} \}
$$
as well as
\begin{equation}\label{d:Gbar}
\bar \cA = \cA \cap \{ b_t \le - \eps \log y\}.
\end{equation}
Hence the goal of \cref{prop:prob} is to control $\P (\bar \cA)$ (note that $\cA$ depends only on the reverse SLE flow). We can use Koebe's 1/4 theorem to get a bound on the variance of $\harm(h \circ f_t)$ on $\cA$:
\begin{lemma}
	Let $(f_s)$ be a reverse $SLE_\kappa$ process with driving function $(\xi_t)$. Let $h$ be an independent Neumann boundary GFF, normalised so that its harmonic part vanishes at the point $iy_0$. Then, for fixed $\varepsilon > 0$ and all $t \in [0,1]$,  and $y \in [0,1]$, if $\cF_t = \sigma(f_s)_{s \le t}$ and $\cF = \cF_1$, on the event $\cA = \cA(t, y , \eps)$:
	\begin{equation*}
		\mb E \left[ \harm(h \circ f_t)(iy)^2 | \cF \right] \leq C' -7\varepsilon\log y,
	\end{equation*}
	where $C'$ is a constant depending on the pinned point $y_0 > \sqrt{5}$ only.
	\label{lem:variance.bound}
\end{lemma}

\begin{proof}
	By conformal invariance of harmonic functions,
	\begin{equation*}
		\harm(h\circ f_t)(\cdot) = \harmt(h)(f_t(\cdot)).
	\end{equation*}
	Since $y_0 > \sqrt{5}$ (recall \eqref{L:bound_im}), we can apply Lemma \ref{lem:maybe} to see that
	\begin{align*}
		\mb E\left[ \harm(h \circ f_t)(iy)^2  | \cF \right] &
	\leq \left( -3\log(\dist(f_t(iy), \partial H_t)) + 2\log(\Re(f_t(iy))^2 + 4y_0^2) + C \right).
	\end{align*}
Now, by Koebe's 1/4 theorem,
	\begin{equation}
		 \dist(f_t(iy), \partial H_t) \ge (1/4) |f_t'(iy)| \dist(iy, \partial \mb H)  = ( | y | /4) | f'_t(iy)|.\label{eq:koebe}
    \end{equation}
    Hence on $\cA$, using Lemma \ref{lem:real.bound}, we get
	\begin{align}\label{eq:complicated}
		\nonumber \mb E\left[ \harm(h \circ f_t)(iy)^2 | \cF \right] & \leq 	
-3\log \left( \frac{y^\varepsilon}{4} \right) + 2 \log\left( 4y^{-2\varepsilon} + 4y_0^2 \right) + C\\
&\leq -3\log ({y^\varepsilon}) + 2\log(8y^{-2\varepsilon}) + 2\log(8y_0^2) + C \nonumber\\
& \leq - 7 \eps \log y + C,
	\end{align}
%	We know that $y^{-2\varepsilon} \geq 1$ and $y_0 > 1$, and so we can see that
%	\begin{equation*}
%		\log(y^{-2\varepsilon} + 4 y_0^2) \leq \log(5y^{-2\varepsilon}) + %\log(5y_0^2).
%	\end{equation*}
%	Therefore, we can simplify the bound in \eqref{eq:complicated} to
%	\begin{equation*}
%		\mb E_h\left[ \harm(h \circ f_t)(iy)^2 \ind_{ \{ |f'_t(iy)| > %y^{-(1-\varepsilon)} \} } \ind_{ \{ \sup_{t \in [0,1]}\xi_t \leq %y^{-\varepsilon} \} } \right] \leq -7 \varepsilon \log y + C',
%	\end{equation*}
where $C$ depends only on $y_0$, as desired.
\end{proof}

\begin{lemma}
Consider the setup of the Liouville quantum gravity coupling in Corollary \ref{cor:coord}, and for $a > 1$ let $b = a/(a-1)$ be its H\"older conjugate.
%
%	Let $(f_s)$ be a reverse $SLE_\kappa$ with driving function $(\xi_s)$. Fix $t \in [0,1]$, $\varepsilon > 0$ and $y > 0$. Let $a > 1$ and $b = a/(a-1)$ be its H\"older conjugate. Let $h$ be a Neumann boundary GFF independent of $(f_s)$ and write $\mb E_h$ for expectation with respect to the law of $h$, conditionally on $(\xi_s)$. Finally, let $h'$ be a Neumann boundary GFF as defined by the coordinate change in Corollary \ref{cor:coord}. Define the event $A$ by
%	\begin{equation*}
%		A := \{|f_t'(iy)| > y^{-(1-\varepsilon)}, \: b_t \leq -\varepsilon \log y, \: \sup_{t \in [0,1]}|\xi_t| \leq y^{-\varepsilon} \}.
%	\end{equation*}
	Then on $\bar \cA$,
	\begin{equation*}
		|f_t'(iy)|^a  \leq  C  y^{\frac{2a}{Q\sqrt \kappa} -  C' \varepsilon} \mb E\left[ \exp\left( \frac{a}{Q}\left( \harm(h')(iy) + b_t \right) \right) | \cF \right],
	\end{equation*}
	where $h = h'$ in law as distributions (they are both normalised to have zero harmonic part at $iy_0$), and the constants $ C,C'$ depend only on $\kappa$, the power $a$ and the pinned point $y_0$ used in the coupling.
	\label{lem:f.bound}
\end{lemma}

\begin{proof}
Recall that by Corollary \ref{cor:coord},
	\begin{equation}
		|f_t'(iy)| =
		\left( \frac{y}{|f_t(iy)|} \right)^{\frac{2}{Q\sqrt \kappa}} \exp\left( \frac{1}{Q}\left( \harm(h')(iy) - \harm(h \circ f_t)(iy) + b_t \right) \right).
		\label{eq:change.A}
	\end{equation}
	We already know from Koebe's 1/4 theorem (see \eqref{eq:koebe}) that
$$
\dist(f_t(iy), \partial H_t) \ge ( | y | /4) | f'_t(iy)| .$$
 On the other hand,
$$ \dist(f_t(iy), \partial H_t) \le \dist ( f_t(iy) , \partial \H) = \Im ( f_t(y) ) \le | f_t(iy)|$$
so that on $\bar \cA$, $|f_t(iy)| \ge y^\eps/4$.
Consequently,
$
		{y}/{|f_t(iy)|} \leq 4y^{1-\varepsilon}.
$	Furthermore by definition of $\bar \cA$, $b_t \leq -\varepsilon \log y$, hence
	\begin{equation}
		\exp\left( \frac{1}{Q}b_t \right)\ind_{\bar \cA} \leq y^{-\varepsilon/Q}.
		\label{eq:b.bound}
	\end{equation}
	Substituting into \eqref{eq:change.A} gives the inequality, on $\bar \cA$:
	\begin{equation*}
		|f'_t(iy)| \leq (4y^{1-\varepsilon})^{\frac{2}{Q\sqrt \kappa}} \cdot y^{-\varepsilon/Q} \exp\left( \frac{1}{Q}(\harm(h')(iy) - \harm(h \circ f_t)(iy)) \right)
	\end{equation*}
	We take the conditional expectation given $\cF$ and use H\"older's inequality (conditionally), taking care where we put the indicator function, and obtain:
	\begin{align}
	|f_t'(iy)| \ind_{\bar \cA} \leq
		&%\hspace{1em}\leq
\left( 4 y^{1-C\varepsilon } \right)^{\frac{2}{Q\sqrt \kappa}} \mb E \left[ \exp\left( \frac{a}{Q}(\harm(h')(iy)) | \cF\right) \right]^{\frac{1}{a}} \mb E \left[ \exp\left( -\frac{b}{Q}\harm(h\circ f_t)(iy) \right) \ind_{\bar \cA} | \cF\right]^{\frac{1}{b}}.
		\label{eq:holder2}
	\end{align}
	Now we know that, conditionally on $\cF$,  $\harm(h\circ f_t)(iy)$ is Gaussian, with a variance that can be bounded by Lemma \ref{lem:variance.bound}. Hence
	\begin{align}
		\mb E \left[ \exp\left( -\frac{b}{Q}\harm(h\circ f_t)(iy) \right) \ind_{\bar \cA} | \cF\right] &\leq
		\exp\left( \frac{b^2}{2Q^2} \cdot \left( C' - 7\varepsilon \log y \right) \right)
		= e^{\frac{b^2 C'}{2Q^2}} y^{-\varepsilon \frac{7b^2}{2Q^2}}.
		\label{eq:4.bound}
	\end{align}
	Substituting \eqref{eq:4.bound} into \eqref{eq:holder2} gives
	\begin{align*}
		|f_t'(iy)| \ind_{\bar \cA}
		&\leq \left(4y^{1-C\varepsilon %\left( 1 + \frac{\sqrt \kappa}{2} \right)
}\right)^{\frac{2}{Q\sqrt \kappa}} e^{\frac{b C'}{2Q^2}} y^{-\varepsilon \frac{7b}{2Q^2}}\mb E \left[ \exp\left( \frac{a}{Q}(\harm(h')(iy) ) \right) | \cF\right]^{\frac{1}{a}}\\
		& = %4^{\frac{2}{Q\sqrt\kappa}} e^{\frac{b C'}{2Q^2}}
C y^{\frac{2}{Q\sqrt \kappa} - C'\varepsilon}\mb E \left[ \exp\left( \frac{a}{Q}(\harm(h')(iy)) \right) | \cF \right]^{\frac{1}{a}}
	\end{align*}
	We finish the proof by raising everything to the power $a$.
\end{proof}

\begin{lemma}\label{cor:mug}
	In the same setting as Lemma \ref{lem:f.bound}, we have for any $a>1$,
	\begin{equation*}
    \mb P \left[ \bar \cA \right]
		\leq C y^{a\left( 1 + \frac{2}{Q\sqrt \kappa} - \frac{a}{Q^2} \right) -  C' \varepsilon},
	\end{equation*}
	where $C, C'$ are constants which depend only on $\kappa$, the H\"older exponent $a$ and the pinned point $iy_0$.
\end{lemma}
\begin{proof}
%	First, recall the definition of the set $A$ from Lemma \ref{lem:f.bound}, that
%	\begin{equation*}
%		A := \{ |f_t'(iy)| > y^{-(1-\varepsilon)}, \: b_t \leq -\varepsilon \log y, \: \sup_{t \in [0,1]} \xi_t \leq y^{-\varepsilon} \}.
%	\end{equation*}
We proceed as in Markov's inequality, using Lemma \ref{lem:f.bound}:
	\begin{align}
		 \P \left[ \bar \cA\right]
%\mb E \left[ \ind_{ \{ |f_t'(iy)| > y^{-(1-\varepsilon)}, \: b_t \leq -\varepsilon \log y, \: \sup_{t \in [0,1]} \xi_t \leq y^{-\varepsilon} \} } \right]\\
		%%
& \nonumber \leq \mb E \left[ y^{a(1-\varepsilon)}|f_t'(iy)|^a \ind_{\bar \cA} \right]\\
& \nonumber \leq \mb E \left[ y^{a(1-\varepsilon)} C y^{\frac{2a}{Q\sqrt \kappa} - C' \varepsilon} \mb E\left[ \exp\left( \frac{a}{Q} \harm(h')(iy) \right) | \cF \right] \right]\\
		&= C y^{a\left( 1 + \frac{2}{Q\sqrt \kappa} \right) - \varepsilon( C' + a)} \mb E \left[ \exp\left( \frac{a}{Q}\harm(h')(iy) \right) \right].
		\label{eq:almost}
	\end{align}
	Now, we use \cref{cor:imag.cov} to bound the variance of $\harm(h')(iy)$ (since $h$ and $h'$ have the same law as normalised distributions), and get
	\begin{align}
\mb E \left[ \exp\left( \frac{a}{Q}\harm(h')(iy) \right) \right]
		&\leq \exp \left( \frac{a^2}{2Q^2}\left( 2\log y_0 - 2\log y \right) \right) =  y_0^{\frac{a^2}{Q^2}} y^{-\frac{a^2}{Q^2}}.
		\label{eq:things}
	\end{align}
	So, substituting \eqref{eq:things} into \eqref{eq:almost} we find
	\begin{equation*}
		\mb P \left[ |f_t'(iy)| > y^{-(1-\varepsilon)}, \: b_t \leq -\varepsilon \log y, \: \sup_{t \in [0,1]} |\xi_t |\leq y^{-\varepsilon} \right] \leq C y^{a\left( 1 + \frac{2}{Q\sqrt \kappa} - \frac{a}{Q^2} \right) - C' \varepsilon},
	\end{equation*}
	as desired.
\end{proof}

To conclude the proof of Proposition \ref{prop:prob} it remains to optimise over $a>1$.

\begin{proof}[Proof of Proposition \ref{prop:prob}]
We focus on the exponent of $y$	in \cref{cor:mug}, which is
\begin{equation}
		f(a) = a\left( 1 + \frac{2}{Q\sqrt \kappa} - \frac{a}{Q^2} \right).
		\label{eq:exponentfinal}
	\end{equation}
	This is a quadratic in $a$, with roots at $a = 0$ and $a = Q^2\left( 1 + \frac{2}{Q\sqrt \kappa} \right)$, and so achieves its maximum at the average of these two, at $a_{\max} = \frac{Q^2}{2}\left( 1 + \frac{2}{Q\sqrt \kappa} \right)$ (note that $Q>2$ so $a_{\max} > 1$.) Substituting $a_{\max}$ into \eqref{eq:exponentfinal} gives
	\begin{equation}
		f(a_{\max}) = a_{\max}\left( 1 + \frac{2}{Q\sqrt \kappa} - \frac{a_{\max}}{Q^2} \right) = \frac{Q^2}{4}\left( 1 + \frac{2}{Q\sqrt \kappa} \right)^2 =  \frac{4}{\kappa} + \frac{\kappa}{16} + 1 = q
		\label{eq:max.exp}
	\end{equation}
after some elementary computations.
%Now note that
%\begin{equation*}
%	Q = \frac{2}{\sqrt \kappa} + \frac{\sqrt \kappa}{2}, \qquad %\frac{Q}{\sqrt \kappa} = \frac{2}{\kappa} + \frac{1}{2}, \qquad Q^2 = %\frac{4}{\kappa} + \frac{\kappa}{4} + 2.
%\end{equation*}
%Therefore, elementary computations show
%\begin{align*}
%	f(a_{\max}) = \frac{Q^2}{4}\left( 1 + \frac{2}{Q \sqrt \kappa} \right)^2 %&= \frac{Q^2}{4}\left( 1 + \frac{4}{Q \sqrt \kappa} + \frac{4}{Q^2 \kappa} \right) \\
%	%%
%	&= \frac{Q^2}{4} + \frac{Q}{\sqrt \kappa} + \frac{1}{\kappa}\\
%	%%
%	&= \frac{1}{\kappa} + \frac{\kappa}{16} + \frac{1}{2} + \frac{2}{\kappa} + %\frac{1}{2} + \frac{1}{\kappa} \\
%	%%
%	&= \frac{4}{\kappa} + \frac{\kappa}{16} + 1 = q,
%\end{align*}
Hence by \cref{cor:mug}, we get
$$
\P( \bar \cA) \le C y^{q - C'\eps}.
$$
%%%%%%%%%%%%%
\begin{comment}
Differentiating to find a minimum, we find
\begin{equation*}
	\frac{\partial}{\partial \kappa}\left( \frac{4}{\kappa} + \frac{\kappa}{16} + 1 \right) = -\frac{4}{\kappa^2} + \frac{1}{16} \qquad \text{and so} \qquad \frac{2}{\kappa} = \frac{1}{4}.
\end{equation*}
Therefore, the exponent is minimised at $\kappa = 8$, at which point the exponent takes the value
\begin{equation*}
	\frac{4}{8} + \frac{8}{16} + 1 = 2.
\end{equation*}

So we see that, for $\kappa \neq 8$, we can choose $\varepsilon > 0$ small enough (depending on $\kappa$ but not $y$) that there exists some $\delta > 0$ such that
\begin{align*}
	\mb P \left[ |f_t'(iy)| > y^{-(1-\varepsilon)}, \: b_t \leq -\varepsilon \log y, \: \sup_{t \in [0,1]} \xi_t \leq y^{-\varepsilon} \right]
	%
	&\leq C y^{a_{max}\left( 1 + \frac{2}{Q \sqrt \kappa} - \frac{a_{max}}{Q^2} \right) - \overline C \varepsilon}  \\
	&= C y^{2 + \delta},
\end{align*}
\end{comment}
%%%%%%%%
which
finishes the proof of \cref{prop:prob1}, and thus \cref{T:mainq}.
\end{proof}

\bibliography{RGM}
\bibliographystyle{plain}
\end{document}